\documentclass[11pt]{amsart}

\usepackage{amsfonts, amsmath, amscd}
\usepackage[psamsfonts]{amssymb}

\usepackage{amssymb}

\usepackage[usenames]{color}

\newtheorem{theorem}{Theorem}[section]

\newtheorem{corollary}[theorem]{Corollary}

\newtheorem{remark}[theorem]{Remark}
\newtheorem{proposition}[theorem]{Proposition}
\newtheorem{definition}[theorem]{Definition}

\newtheorem{fact}[theorem]{Fact}

\newtheorem{problem}[theorem]{Problem}

\numberwithin{equation}{section}

\newcommand{\CC}{C_k}
\newcommand{\NN}{\mathbb{N}}

\newcommand{\GG}{\mathfrak{G}}

\newcommand{\UU}{\mathcal{U}}

\newcommand{\w}{\omega}
\newcommand{\e}{\varepsilon}

\newcommand{\KK}{\mathcal{K}}

\newcommand{\AAA}{\mathcal A}

\renewcommand{\phi}{\varphi}

\input xy
\xyoption{all}

\title[Free locally convex spaces with a small base] {Free locally convex spaces  with a small base}

\author[S. Gabriyelyan]{Saak Gabriyelyan}
\address{Department of Mathematics, Ben-Gurion University of the Negev, Beer-Sheva, P.O. 653, Israel}
\email{saak@math.bgu.ac.il}

\author[J. K{\c{a}}kol]{Jerzy K{\c{a}}kol}
\address{A. Mickiewicz University $61-614$ Pozna{\'n}, Poland and Institute of Mathematics, Czech Academy of Sciences, Czech Republic}
\email{kakol@amu.edu.pl}

\subjclass[2000]{Primary 46A03; Secondary 54A25, 54D50}

\keywords{free locally convex space, $\GG$-base, $\CC(X)$, compact resolution}
\thanks{The research was supported for the second named author  by Generalitat Valenciana, Conselleria d'Educaci\' o i Esport, Spain, Grant PROMETEO/2015/058   and by the GA\v{C}R project 16-34860L and RVO: 67985840. The second author  gratefully acknowledges also the financial support he received from the Center for Advanced Studies in Mathematics of the Ben Gurion University of the Negev during his visit Marz 15 - 22, 2016.}

\begin{document}

\begin{abstract}
The paper studies  the free locally convex space $L(X)$ over  a Tychonoff space $X$. Since for infinite $X$ the space $L(X)$ is never metrizable (even not Fr\'echet-Urysohn), a possible applicable generalized metric property for $L(X)$ is welcome. We propose a concept (essentially weaker than  first-countability) which is known under the name  a $\GG$-base. A space $X$ has a {\em $\GG$-base} if for every $x\in X$ there is a base $\{ U_\alpha : \alpha\in\NN^\NN\}$ of neighborhoods at $x$ such that $U_\beta \subseteq U_\alpha$ whenever $\alpha\leq\beta$ for all $\alpha,\beta\in\NN^\NN$, where $\alpha=(\alpha(n))_{n\in\NN}\leq \beta=(\beta(n))_{n\in\NN}$ if $\alpha(n)\leq\beta(n)$ for all $n\in\NN$.
We show that if $X$ is an Ascoli $\sigma$-compact space, then $L(X)$ has a $\GG$-base if and only if $X$ admits an Ascoli uniformity $\UU$ with a $\GG$-base. We prove that if $X$ is a $\sigma$-compact Ascoli space of $\NN^\NN$-uniformly compact type, then $L(X)$  has a $\GG$-base. As an application we show: (1) if  $X$ is  a metrizable space, then $L(X)$ has a $\GG$-base if and only if $X$ is $\sigma$-compact, and (2) if $X$ is a countable Ascoli space, then $L(X)$ has a $\GG$-base if and only if $X$ has a $\GG$-base.
\end{abstract}


\maketitle


\section{Introduction}

\bigskip
The class of free locally convex spaces $L(X)$ over a (Tychonoff) space $X$ is one of the most important classes in the category of locally convex spaces and continuous operators. This class was introduced by Markov \cite{Mar} and intensively studied over the last half-century, see for example \cite{ArT,Flo2,Gab-MSJ,Rai,Usp2}.
Recall that the {\em  free locally convex space} $L(X)$ over a  space $X$ is a pair consisting of a locally convex space $L(X)$ and  a continuous mapping $i: X\to L(X)$ such that every  continuous mapping $f$ from $X$ to a locally convex space $E$ gives rise to a unique continuous linear operator ${\bar f}: L(X) \to E$  with $f={\bar f} \circ i$. The free locally convex space $L(X)$  always exists and is  unique.

It is well-known that $L(X)$ is metrizable if and only if $X$ is finite. Moreover, $L(X)$ is a $k$-space if and only if $X$ is a countable discrete space, see \cite{Gabr}. Therefore,  seeking for concrete objects $L(X)$  carrying   some \emph{small base}
 at zero might be interesting for specialist both from topology and functional analysis.

One of such possible  concepts extending metrizability is related with locally convex spaces having a $\GG$-base. Following \cite{GKL}, a topological space $X$ has a {\em $\GG$-base at a point $x\in X$} if it has a base $\{ U_\alpha : \alpha\in\NN^\NN\}$ of neighborhoods at $x$ such that $U_\beta \subseteq U_\alpha$ whenever $\alpha\leq\beta$ for all $\alpha,\beta\in\NN^\NN$, where $\alpha=(\alpha(n))_{n\in\NN}\leq \beta=(\beta(n))_{n\in\NN}$ if $\alpha(n)\leq\beta(n)$ for all $n\in\NN$;  $X$ has a {\em $\GG$-base} if it has a $\GG$-base at each point $x\in X$.

Originally, the concept of a $\GG$-base has been formally introduced in \cite{FKLS} in the realm of locally convex spaces for studying $(DF)$-spaces, $C(X)$-spaces and spaces in  the class $\GG$ in the sense of Cascales and Orihuela, see \cite{kak}.  Every quasibarrelled locally convex space  with a $\GG$-base has countable tightness both in the original and the weak topology, respectively; each precompact set in a locally convex space with a $\GG$-base is metrizable, see again \cite{kak}. It is easy to see that every metrizable group has a $\GG$-base at the identity. Topological groups with a $\GG$-base are thoroughly studied in \cite{GKL}, see also \cite{CFHT,GK-GMS2,GKL2}.

Being motivated by several results of the above type (see \cite{kak} also for a long list of references), the authors in \cite{GKL} posed the following general problem:
\begin{problem}[\cite{GKL}] \label{p:Free-G-base}
For which  spaces $X$ the free locally convex space $L(X)$ has a $\GG$-base?
\end{problem}
For a  space $X$ we denote by $C_p(X)$ and $\CC(X)$ the space $C(X)$ of all continuous real-valued functions on $X$ endowed with the pointwise topology $\tau_p$ and the compact-open topology $\tau_k$, respectively.
 Recall that a space $X$ is called an {\em  Ascoli space} if every  compact subset $\KK$ of $\CC(X)$ is evenly continuous \cite{BG}. In other words, $X$ is Ascoli if and only if the compact-open topology of $\CC(X)$ is Ascoli in the sense of \cite[p.45]{mcoy}.

Using a deep result of Uspenski\u{\i} \cite{Usp2}, for a wide class of topological spaces
$X$ we show that  Problem \ref{p:Free-G-base} can be reformulated in the term of  function spaces $C(X)$. 

\begin{theorem} \label{t:Free-G-base-Main}
Let $X$ be a Dieudonn\'{e} complete  Ascoli space (in particular, $X$ is a paracompact $k$-space or a metrizable space). Then  $L(X)$ has a $\GG$-base if and only if  $\CC(X)$ has a compact resolution swallowing compact subsets.
\end{theorem}
Recall that a family $\KK =\{ K_\alpha : \alpha\in \NN^\NN\}$ of compact subsets of a space $Z$ is called a {\em compact resolution} if $\KK$ covers $Z$ and $K_\alpha  \subseteq  K_\beta$ whenever $\alpha\leq\beta$ for all $\alpha,\beta\in\NN^\NN$.  Following Christensen \cite{Chri}, we say that $\KK$ {\em swallows compact sets of $Z$} if for every compact subset $K$ of $Z$ there is an $\alpha\in\NN^\NN$ such that $K\subseteq K_\alpha$. The importance of this concept  follows from the following deep result of Christensen: \emph{A metrizable and separable space $Z$ is Polish if and only if $Z$ has a compact resolution swallowing compact sets.} Consequently, since $\CC(X)$ is Polish if $X$ is locally compact metrizable and separable,  by Theorem \ref{t:Free-G-base-Main} the space $L(X)$ has a $\GG$-base. These results and Theorem \ref{t:Free-G-base-Main}  motivate the following question:
\begin{problem} \label{p:Free-compact-resol}
For which  spaces $X$, the space $\CC(X)$ has a compact resolution (swallowing  compact sets)?
\end{problem}
This problem is of independent interest because (see for example \cite[Theorem 9.9]{kak})  $\CC(X)$ has a compact resolution if and only if $\CC(X)$ is $K$-analytic, i.e. $\CC(X)$  is the image under an upper semi-continuous compact-valued map defined in $\NN^\NN$; the same result holds for $C_{p}(X)$, see  \cite{Tkachuk}. Moreover,  Tkachuk proved in \cite{Tkachuk}  that $C_p(X)$ has a compact resolution swallowing  compact sets if and only if $X$ is a countable discrete space.

Christensen had already proved the following result (see also Corollary \ref{c:Free-L(X)-metr} below): {\em If $X$ is a separable metrizable space, then   $\CC(X)$  has a compact resolution if and only if $X$ is $\sigma$-compact.} Below we strengthen this result by showing that under the same assumption on $X$ the space $\CC(X)$ has even a compact resolution swallowing  compact sets, see Corollary \ref{c:Free-resol-metric} below. These results motivate the question: {\em For which $\sigma$-compact spaces $X$ the space $\CC(X)$ has a  compact resolution (swallowing  compact sets)}? The aforementioned results explain our study of  functions spaces with compact resolutions, see  Section \ref{sec:2}.

In Section \ref{sec:1} we prove Theorem \ref{t:Free-G-base-Main} and obtain the following partial answers to Problem \ref{p:Free-G-base}.
\begin{theorem} \label{t:Free-G-base-resolution}
Let $X$ be an Ascoli $\sigma$-compact space. Then $L(X)$ has a $\GG$-base if and only if $X$ admits an Ascoli uniformity $\UU$ with a $\GG$-base.
\end{theorem}
Theorem  \ref{t:Free-G-base-resolution} needs a new concept  which is stronger than to be an Ascoli space.
\begin{definition} {\em
A uniformity $\UU$ on a  space $X$ is said to be {\em Ascoli} if $\UU$ is admissible and any compact subset $K$ of $\CC(X)$ is uniformly equicontinuous with respect to $\UU$, i.e. for every $\e>0$ there is $U\in\UU$ such that $|f(x)-f(y)|<\e$ for every $f\in K$ and each $(x,y)\in U$. We say that $X$ is a {\em uniformly Ascoli space} if $X$ has an Ascoli uniformity.}
\end{definition}

We provide  also a sufficient condition on a $\sigma$-compact space $X$ for which $L(X)$ has a $\GG$-base. This approach requires some additional concept.
\begin{definition}  \label{def:Free-uniform-compact}{\em
A topological space $X$ is a space of {\em $\NN^\NN$-uniformly compact type} if for every compact subset $K$ of $X$ the set $\Delta_K =\{ (x,x)\in X\times X: x\in K\}$ has an $\NN^\NN$-decreasing base $\{ U_\alpha: \alpha\in \NN^\NN\}$ of open neighborhoods in $X\times X$, i.e. for every open neighborhood $U$ of $\Delta_K$ there is $\alpha\in \NN^\NN$ such that $\Delta_K \subseteq U_\alpha \subseteq U$.}
\end{definition}

\begin{theorem} \label{t:Free-LCS-G-base}
Let $X$ be a $\sigma$-compact Ascoli space. If $X$ is of $\NN^\NN$-uniformly compact type, then $L(X)$  has a $\GG$-base.
\end{theorem}
It is easy to show, see Proposition \ref{p:Free-LCS-G-base} below, that if $X$ is a metrizable space or  a countable space such that every point $x\in X$ has a $\GG$-base, then $X$ is of $\NN^\NN$-uniformly compact type. So Theorem \ref{t:Free-LCS-G-base} with
 Corollary \ref{c:Free-resol-metric} imply
\begin{corollary} \label{c:Free_G-base-metric}
If $X$ is  a metrizable space, then $L(X)$ has a $\GG$-base if and only if $X$ is $\sigma$-compact.
\end{corollary}
 In particular, the space $L(\mathbb{Q})$ has a $\GG$-base.
\begin{corollary} \label{c:Free_G-base-countable}
If $X$ is a countable Ascoli space, then $L(X)$ has a $\GG$-base if and only if $X$ has a $\GG$-base.
\end{corollary}
Note that Corollaries \ref{c:Free_G-base-metric} and \ref{c:Free_G-base-countable} are proved independently in \cite{BL} using  different methods.


\section{Compact resolutions in function spaces} \label{sec:2}

Recall that a subset $A$ of a topological space $X$ is called {\em functionally bounded} if every continuous function $f\in C(X)$ is bounded on $A$.
Recall also that a {\em resolution} $\AAA =\{ A_\alpha: \alpha\in\NN^\NN \}$ in $X$ is a cover of $X$ such that $A_\alpha  \subseteq  A_\beta$ whenever $\alpha\leq\beta$ for all $\alpha,\beta\in\NN^\NN$. If all $A_\alpha$ are functionally bounded, the resolution $\AAA$ is called {\em  functionally bounded}. Note also that by  a result of Calbrix, see \cite[Theorem 9.7]{kak}, if $C_{p}(X)$ is analytic, then $X$ is $\sigma$-compact.
Recall that a space $Z$ is called {\em analytic} if it is a continuous image of $\NN^\NN$.

We shall use the following fact, see Corollary 9.1 of \cite{kak}.
Recall that a (Tychonoff) space $X$ is {\em cosmic} if it is a continuous image of a separable metric space.
\begin{fact} \label{f:Free-G-base-cosmic}
Let $X$ be a cosmic space. Then $C_p(X)$ has a functionally bounded resolution if and only if  $X$ is $\sigma$-compact. Consequently,  if $\CC(X)$ has a compact resolution, then $X$ is $\sigma$-compact.
\end{fact}

\begin{proposition} \label{p:Free-angelic-Lin}
Let $X$ be a paracompact first countable space such that $\CC(X)$ is an angelic space. Then $X$ is Lindel\"{o}f.
\end{proposition}
\begin{proof}
If $X$ is not Lindel\"{o}f, the space $\CC(X)$ contains  a closed subset $A$ homeomorphic to $\NN^{\w_1}$ by  Lemma 1 of \cite{Pol-1974}. But the space $\NN^{\w_1}$ is not angelic, so $\CC(X)$ is also not angelic. This contradiction shows that $X$ must be  Lindel\"{o}f.
\end{proof}

This yields the following
\begin{corollary} \label{c:Free-L(X)-metr}
Let $X$ be a metrizable space. If $\CC(X)$ has a functionally bounded resolution, then $X$ is $\sigma$-compact.
\end{corollary}
\begin{proof}
Proposition 9.6 of \cite{kak} implies that $C_p(X)$ is angelic. By  \cite[Theorem, page 31]{Floret}, the space $\CC(X)$ is also angelic. Now Proposition \ref{p:Free-angelic-Lin} implies that $X$ is Lindel\"{o}f. So being metrizable, the space $X$ is separable, and hence $X$ is a cosmic space. Thus $X$ is $\sigma$-compact by Fact \ref{f:Free-G-base-cosmic}.
\end{proof}

The next proposition completes Proposition \ref{p:Free-angelic-Lin}.
\begin{proposition} \label{p:Free-Cech-angelic-Lin}
Let $X$ be a paracompact \v{C}ech-complete  space. Then $\CC(X)$ is an angelic space if and only if $X$ is Lindel\"{o}f.
\end{proposition}
\begin{proof}
Assume that $\CC(X)$ is angelic. By a result of Frol\'{i}k \cite[5.5.9]{Eng}, there is a perfect map $f$ from $X$ onto a complete metrizable space $Y$. Suppose that $X$ is not  Lindel\"{o}f. Then, since $f$ is perfect, $Y$ is also not  Lindel\"{o}f by \cite[Theorem 3.8.9]{Eng}.
As $f^{-1}(K)$ is compact for every compact set $K\subseteq Y$ by \cite[Theorem 3.7.2]{Eng}, the space $\CC(Y)$ embeds into $\CC(X)$, and hence $\CC(Y)$ is also angelic. Now Proposition \ref{p:Free-angelic-Lin} implies that $Y$ is  Lindel\"{o}f, a contradiction. Thus $X$ is a  Lindel\"{o}f space.

Conversely, let $X$ be  Lindel\"{o}f. Then $X$ has a compact resolution swallowing compact sets, see the proof of Proposition 4.7 in \cite{GKKLP}. So $C_p(X)$ is angelic by Example 4.1 and Theorem 4.5 of \cite{kak}. Therefore $\CC(X)$ is angelic by  \cite[Theorem, page 31]{Floret}.
\end{proof}

Recall that, for a  space $X$, the family of sets of the form
\[
[K,\e] :=\{ f\in C(X): f(K)\subset (-\e,\e)\}
\]
where $K$ is a compact subset of $X$, is a base of the compact-open topology $\tau_k$ on $C(X)$.
Denote by $\delta: X\to \CC(\CC(X))$ the canonical map defined by
\[
\delta(x)(f):= f(x), \quad  \forall x\in X, \; \forall f\in C(X).
\]

\begin{proposition} \label{p:Free-G-base-ascoli}
Let $X$ be an Ascoli space. If $\CC(X)$ has a compact resolution swallowing  compact sets, then $X$ has an Ascoli uniformity $\UU$ with a $\GG$-base.
\end{proposition}
\begin{proof}
Let $\KK :=\{ K_\alpha: \alpha\in\NN^\NN\}$ be a compact resolution swallowing  compact sets of $\CC(X)$. Then, by \cite{feka}, the space $\CC\big( \CC(X)\big)$ has a $\GG$-base $\{ V_\alpha: \alpha\in\NN^\NN\}$, where $V_\alpha := \big[K_\alpha, \alpha(1)^{-1}\big]$.
Since $X$ is Ascoli space, the canonical map $\delta : X\to \CC\big( \CC(X)\big)$ is an embedding by Corollary 5.8 of \cite{BG}. For every $\alpha\in\NN^\NN$, define
\[
U_\alpha :=\{ (x,y)\in X\times X: \; \delta(x)-\delta(y) \in V_\alpha\},
\]
and let $\UU$ be the uniformity on $X$ induced from $\CC\big( \CC(X)\big)$. Clearly, $\{ U_\alpha: \alpha\in\NN^\NN\}$ is a $\GG$-base for $\UU$ and $\UU$ is admissible. We show that $\UU$ is also Ascoli.

Fix a compact subset $K$ of $\CC(X)$ and $\e>0$. Take $\alpha\in \NN^\NN$ such that $K\subseteq K_\alpha$ and $\alpha(1) >1/\e$. Now for every $(x,y)\in U_\alpha$ and each $f\in K$, we obtain
\[
|f(x)-f(y)|=\big| \big( \delta(x)-\delta(y)\big)(f)\big|< \frac{1}{\alpha(1)} <\e,
\]
Thus $\UU$ is an Ascoli uniformity.
\end{proof}

We shall use the following encoding operation of elements of $\NN^\NN$. We encode each $\alpha\in\NN^\NN$ into a sequence $\{ \alpha_i\}_{i\in\w}$ of elements of $\NN^\NN$ as follows. Consider an arbitrary decomposition of $\NN$ onto a disjoint family $\{ N_i\}_{i\in\w}$ of infinite sets, where $N_i =\{ n_{k,i} \}_{k\in\NN}$. Now for $\alpha=\big(\alpha(n)\big)_{n\in\NN}$ and $i\in\w$, we set $\alpha_i=\big(\alpha_i(k)\big)_{k\in\NN}$, where $\alpha_i(k):= \alpha(n_{k,i})$ for every $k\in\NN$. Conversely, for every sequence $\{ \alpha_i\}_{i\in\w}$ of elements of $\NN^\NN$ we define $\alpha=\big(\alpha(n)\big)_{n\in\NN}$ setting $\alpha(n):= \alpha_i(k)$ if $n=n_{k,i}$.

For a subset $A$ of a set $S$, a subset $B$  of $S\times S$ and $(a,b)\in B$, we define
\[
\Delta_A:=\{ (a,a)\in S\times S: a\in A\} \mbox{ and } B(a):=\{ s\in S: (a,s)\in B\}.
\]

Next definition  generalizes the classical notion of spaces of pointwise countable type (due to Arhangel'skii) and also Definition \ref{def:Free-uniform-compact}.
\begin{definition}{\em
Let $I$ be an ordered set and $X$ a topological space. The space $X$ is a space of
\begin{enumerate}
\item[(i)] {\em $I$-compact type} if every compact subset $K$ of $X$ has an decreasing $I$-base $\{ U_i: i\in I\}$ of open neighborhoods, i.e. $U_i\subseteq U_j$ for all $i\geq j$ and for every open neighborhood $U$ of $K$ there is $i\in I$ such that $K\subseteq U_i \subseteq U$;
\item[(ii)]  {\em $I$-pointwise countable type} if for every $x$ in $X$ there exists a compact set $K$ which has  a decreasing $I$-base of open sets;
\item[(iii)]  {\em $I$-uniformly compact type} if for every compact subset $K$ of $X$ the set $\Delta_K$ has an $I$-decreasing base $\{ U_i: i\in I\}$ of open neighborhoods in $X\times X$, i.e. for every open neighborhood $U$ of $\Delta_K$ there is $i\in I$ such that $\Delta_K \subseteq U_i \subseteq U$.
\end{enumerate} }
\end{definition}
As usual a decreasing $\NN^\NN$-base of a subset $A$ of $X$ is called a {\em $\GG$-base} of $A$.
Next proposition provides possible  two cases when $X$ is of $\NN^\NN$-uniformly compact type.
\begin{proposition} \label{p:Free-LCS-G-base}
A Tychonoff space $X$ is  of $\NN^\NN$-uniformly compact type if one of the following conditions holds:
\begin{enumerate}
\item[(i)] $X$ is a metrizable space;
\item[(ii)] $X$ is a countable space such that every point $x\in X$ has a $\GG$-base.
\end{enumerate}
\end{proposition}
\begin{proof}
(i) For every compact subset $K$ of $X$, the compact subset $\Delta_K$ of the metrizable space $X\times X$ has a decreasing base $\{ V_n: n\in\NN\}$. Then the family $\{ U_\alpha = V_{\alpha(1)}: \alpha\in \NN^\NN\} $ is a $\GG$-base of $K$.

(ii) Let $\{ x_n: n\in\NN\}$ be an enumeration of $X$ and $\{ U_{\alpha,x_n}: \alpha\in\NN^\NN\}$ be a $\GG$-base at $x_n$. Fix a compact subset $K$ of $X$. If $K$ is finite, then clearly the family
\[
\left\{ \bigcup_{x_n\in K} U_{\alpha,x_n}\times U_{\alpha,x_n}: \alpha\in\NN^\NN\right\}
\]
is a $\GG$-base at $\Delta_K$.

Assume that $K$ is infinite. For every $\alpha\in\NN^\NN$ with the encoding $(\alpha_n)$, we define
\[
U_\alpha :=  \bigcup \{ U_{\alpha_n,x_n}\times U_{\alpha_n,x_n} : x_n\in K\}.
\]
We claim that the family $\UU= \{ U_\alpha : \alpha\in\NN^\NN\}$ is a $\GG$-base at $\Delta_K$. Indeed, fix an open neighborhood $U$ of $\Delta_K$. For every $x_n\in K$ take $\alpha_n\in\NN^\NN$ such that $U_{\alpha_n,x_n}\times U_{\alpha_n,x_n} \subseteq U$. Now if $\alpha\in\NN^\NN$ is built by the sequence $(\alpha_n)$, we obtain $K\subseteq U_\alpha \subseteq U$. Thus $\UU$ is a $\GG$-base at $\Delta_K$.
\end{proof}

We shall use the following fact which is proved in the ``if'' part of the Ascoli theorem \cite[3.4.20]{Eng}.
\begin{fact} \label{f:Ascoli-Free-Ck}
Let $X$ be a  space and $A$ be an evenly continuous (in particular, equicontinuous) pointwise bounded subset of $\CC(X)$. Then the closure ${\bar A}$ of $A$ in $\tau_k$ is a compact  equicontinuous subset of $\CC(X)$. 
\end{fact}

If an  Ascoli space $X$ is additionally $\sigma$-compact, we can reverse Proposition \ref{p:Free-G-base-ascoli}.
\begin{proposition} \label{p:Free-G-base-uniform}
Let $X$ be a $\sigma$-compact space.  Then the space $\CC(X)$ has a compact resolution swallowing  compact sets if one of the following conditions holds:
\begin{enumerate}
\item[(i)] $X$ has an Ascoli uniformity $\UU$ with a $\GG$-base;
\item[(ii)] $X$ is an Ascoli space of $\NN^\NN$-uniformly compact type.
\end{enumerate}
\end{proposition}

\begin{proof}
Let $X=\bigcup_{n\in\NN} C_n$ be the union of an increasing sequence $\{ C_n\}_{n\in\NN}$ of compact subsets. For the case (i), let $\{ U_\alpha: \alpha\in\NN^\NN\}$ be a $\GG$-base for the Ascoli uniformity $\UU$. For the case (ii), for every $n\in \NN$, let $\{ U_{\alpha,n}: \alpha\in\NN^\NN\}$ be a $\GG$-base of $\Delta_{C_n}$.  For every $\alpha\in\NN$ with the encoding $\{ \alpha_k\}_{k\in\w}$, we define
\[
\begin{split}
A_\alpha & :=\bigcap_{k\in\NN}  \{ f\in C(X): \, |f(x)|\leq \alpha_0(k) \quad \forall x\in C_k \}, \\
B_\alpha & := \bigcap_{n\in\NN} \left\{ f\in C(X): \, |f(x)-f(y)|\leq\frac{1}{n} \quad \forall (x,y)\in U_{\alpha_n}\right\}, \; \mbox{ for case (i)}, \\
B_\alpha & := \bigcap_{n\in\NN} \left\{ f\in C(X): \, |f(x)-f(y)|\leq\frac{1}{n} \quad \forall (x,y)\in U_{\alpha_n,n}\right\}, \;  \mbox{ for case (ii)},
\end{split}
\]
and set $K_\alpha := A_\alpha \cap B_\alpha$. Clearly, $K_\alpha$ is closed in the compact-open topology $\tau_k$ and $K_\alpha \subseteq K_\beta$ for every $\alpha\leq\beta$. Fix $\alpha\in\NN^\NN$. By construction,  $K_\alpha$ is pointwise bounded.
We check that the set $K_\alpha$ is equicontinuous. We distinguish between cases (i) and (ii).

{\em Case (i).}  Given $\e >0$ take $n\in\NN$ such that $n> 1/\e$. Then for every $f\in K_\alpha$, by the definition of $B_\alpha$, we obtain $|f(x)-f(y)|\leq\frac{1}{n} <\e$ whenever $(x,y)\in U_{\alpha_n}$. So $K_\alpha$ is equicontinuous.

{\em Case (ii).}  Fix $x\in X$, so $x\in C_l$ for some $l\in\NN$. Given $\e >0$ take $n>l$ such that $n> 1/\e$. Then for every $f\in K_\alpha$, by the definition of $B_\alpha$, we obtain $|f(x)-f(y)|\leq\frac{1}{n} <\e$ whenever $(x,y)\in U_{\alpha_n,n}$. Since $U_{\alpha_n,n}(x)$ is an open neighborhood of $x$, the set $K_\alpha$ is equicontinuous.

Now in both cases (i) and (ii), Fact \ref{f:Ascoli-Free-Ck}  implies that $K_\alpha$ is a compact subset of $\CC(X)$.

Let us show that the family $\KK :=\{ K_\alpha: \alpha\in\NN^\NN\}$ swallows the compact sets of $\CC(X)$. Fix a compact subset  $K$ of $\CC(X)$. Since
$X$ is an Ascoli space, $K$ is pointwise bounded and equicontinuous. Define $\alpha_0 =\big(\alpha_0(k)\big)_{k\in\NN}$ as follows: for every $k\in\NN$, set
\[
\alpha_0(k):= \left[ \sup\{ |f(x)|: x\in C_k , f\in K\} \right] +1,
\]
where $[t]$ is the integral part of a real number $t$. Again  we distinguish between cases (i) and (ii).

{\em Case (i).}  Since $\UU$ is Ascoli, for every $n\in\NN$, take $\alpha_n\in\NN^\NN$ such that $|f(x)-f(y)|\leq 1/n$ for every $f\in K$ and each $(x,y)\in U_{\alpha_n}$. If $\alpha\in\NN^\NN$ is built by the above procedure we obtain $K\subseteq K_\alpha$.

{\em Case (ii).}  Fix $n\in\NN$. For every $x\in C_n$ take an open neighborhood $U_x$ of $x$ such that $|f(x)-f(y)|\leq 1/2n$ for every $f\in K$ and each $y\in U_x$. Set $W:=\bigcup_{x\in C_n} U_x\times U_x$. Then for every $(z,y)\in W$ there is $x\in C_n$ such that $(z,y)\in U_x\times U_x$ and hence
\[
|f(z)-f(y)|\leq |f(x)-f(z)|+|f(x)-f(y)|\leq \frac{1}{n}, \; \mbox{ for every } f\in K.
\]
Since $X$ is  of $\NN^\NN$-uniformly compact type, we choose $\alpha_n\in\NN^\NN$ such that $\Delta_{C_n} \subseteq U_{\alpha_n,n}\subseteq W$.
If $\alpha\in\NN^\NN$ is built by the sequence $(\alpha_n)$, we obtain $K\subseteq K_\alpha$.

Also now in both cases (i) and (ii) the family $\KK$  swallows the compact sets of $\CC(X)$.
\end{proof}

As a corollary we obtain the following strengthening of Christensen's theorem.
\begin{corollary} \label{c:Free-resol-metric}
For a metrizable space $X$, $\CC(X)$ has a compact resolution swallowing the compact sets of $\CC(X)$ if and only if $X$ is $\sigma$-compact.
\end{corollary}
\begin{proof}
If $\CC(X)$ has a compact resolution swallowing the compact sets of $\CC(X)$, then $X$ is $\sigma$-compact by Corollary \ref{c:Free-L(X)-metr}. The converse assertion follows from Propositions \ref{p:Free-LCS-G-base} and \ref{p:Free-G-base-uniform}.
\end{proof}
In particular, the space $\CC(\mathbb{Q})$ has a compact resolution swallowing its compact sets.


We conclude this section with the following Christensen's type result.
\begin{proposition}\label{p:Compact-res-sigma}
The following assertions are equivalent.
\begin{enumerate}
\item [(i)] $\CC(X)$ is analytic.
\item[(ii)] $\CC(X)$ is K-analytic and $X$ is $\sigma$-compact.
\item[(iii)] $\CC(X)$ has a compact resolution and $X$ is $\sigma$-compact.
\end{enumerate}
If additionally $X$ is first countable, the above conditions are equivalent to
\begin{enumerate}
\item [(iv)] $X$ is metrizable  and $\sigma$-compact.
\end{enumerate}
\end{proposition}
\begin{proof}
First we prove the following claim using  some ideas from \cite{fe-ka}  strongly motivated by Ferrando's Theorem 1  of \cite{Ferrando-Cc}.

{\em Claim. If $X$  is a $\sigma$-compact space, then $C_p(X)$  admits a stronger metrizable locally convex topology.} Indeed, let $X=\bigcup_{n=1}^{\infty } K_{n}$, where $K_{n}$ is a compact subset of $X$ and $K_n\subseteq K_{n+1}$ for every $n\in\NN$. For every $n\in\NN$, define
\begin{equation} \label{equ:Compact-resol}
V_{n}:=\left\{ f\in C\left( X\right) :\sup_{x\in K_{n}}\left| f\left( x\right) \right| \leq \frac{1}{n}\right\}.
\end{equation}
Clearly, $V_{n+1}\subseteq V_{n}$ and $\bigcap_{n=1}^{\infty }V_{n}=\left\{ 0\right\} $, where  $0$ stands for the identically null function on $X$. Note that the sets $V_{n}$ are absorbing since if $g\in C\left( X\right) $, then there is $k\in \mathbb{N}$ such that $\sup_{x\in K_{n}}\left| g\left( x\right) \right| \leq k$, so that $g\in knV_{n}$. Moreover, if
\[
U=\left\{ f\in C\left( X\right) :\max_{1\leq i\leq n}\left| f\left( x_{i}\right) \right| <\epsilon \right\}
\]
and $p\in \mathbb{N}$ is chosen so that $x_{i}\in V_{p}$ for $1\leq i\leq n$ and $p^{-1}<\epsilon $, then $V_{p}\subseteq U$ and clearly $V_{2n}\subseteq
2^{-1}V_{n}$ for each $n\in \mathbb{N}$. This shows that $\left\{ V_{n}:n\in \mathbb{N}\right\} $ is a base of neighborhoods of the origin of a locally convex topology on $C\left( X\right) $ stronger than the pointwise topology. The claim is proved.

The implications (i) $\Rightarrow$ (ii) $\Rightarrow$ (iii) are clear; note  that if $\CC(X)$ is analytic, then Calbrix's  result, see \cite[Theorem 9.7]{kak}, implies that $X$ is $\sigma$-compact.

(iii) $\Rightarrow$ (i): If $X$ is $\sigma$-compact, the claim and (\ref{equ:Compact-resol}) show that the space $C(X)$ admits a metrizable locally convex topology weaker (or equal) to the compact-open topology $\tau_{k}$. As $\CC(X)$ has a compact resolution (by assumption), we apply Talagrand's result, see \cite[Proposition 6.3]{kak}, to deduce that $\CC(X)$ is analytic.

If $X$ first countable, (i) is equivalent to (iv) by \cite[Theorem 5.7.5]{mcoy}.
\end{proof}


\section{Proofs of Theorems \ref{t:Free-G-base-Main}, \ref{t:Free-G-base-resolution} and  \ref{t:Free-LCS-G-base}} \label{sec:1}


It is well-known  that the dual space of $\CC(X)$ is the space $M_c(X)$ of all regular Borel measures on $X$ with compact support.
Denote by $\tau_e$ the topology on $M_c(X)$ of uniform convergence on the  equicontinuous pointwise bounded subsets of $C(X)$. For $A\subseteq \CC(X)$ and $B\subseteq M_c(X)$, we set as usual
\[
A^\circ  = \left\{ \mu\in M_c(X): \; |\mu(f)|\leq 1\; \forall f\in A\right\}, \mbox{ and }
B^\circ  = \{ f\in\CC(X): \; |\mu(f)|\leq 1 \; \forall \mu\in B\}.
\]

\begin{proposition} \label{p:G-L-Mc}
Let $X$ be an Ascoli space. Then $(M_c(X),\tau_e)$ has a $\GG$-base if and only if  $\CC(X)$ has a compact resolution swallowing  compact subsets of $\CC(X)$.
\end{proposition}

\begin{proof}
Assume that $\CC(X)$ has a compact resolution swallowing  compact subsets of $\CC(X)$.
Let $\KK =\{ K_\alpha: \alpha\in\NN^\NN\}$ be a compact resolution swallowing  compact sets of $\CC(X)$. For every $\alpha\in\NN^\NN$, set $U_\alpha := K_\alpha^\circ$. We show that the family $\UU :=\{ U_\alpha: \alpha\in\NN^\NN\} $ is a $\GG$-base in $(M_c(X),\tau_e)$. Indeed,
 every $U_\alpha$ is a neighborhood of zero in $\tau_e$ because $K_\alpha$ is equicontinuous.
Now let $U$ be a neighborhood of zero in  $(M_c(X),\tau_e)$. Take an  equicontinuous pointwise bounded subset $A$ of $C(X)$ such that $A^\circ \subseteq U$. By Fact \ref{f:Ascoli-Free-Ck}, the closure $K$ of $A$ in $\CC(X)$ is compact. So there is $\alpha\in\NN^\NN$ such that $K\subseteq K_\alpha$. Clearly, $U_\alpha =K_\alpha^\circ \subseteq A^\circ \subseteq U$. Thus $\UU$ is a base of $\tau_e$.

Conversely, let $(M_c(X),\tau_{e})$ have a $\GG$-base $\{ U_\alpha : \alpha\in\NN^\NN\}$. For every $\alpha\in\NN^\NN$, set $C_\alpha :=U_\alpha^\circ$. We show that the family $\mathcal{C}:=\{ C_\alpha : \alpha\in\NN^\NN\}$ is a compact resolution in $\CC(X)$ swallowing the compact sets.

Clearly, if $\alpha\leq\beta$ then $C_\alpha\subseteq C_\beta$. Since $M_c(X)$ is the dual space of $\CC(X)$, every $C_\alpha$ is closed in $\CC(X)$. To show that $C_\alpha$ is compact in $\CC(X)$, take an absolutely convex neighborhood $V$ of zero in $M_c(X)$ such that $\overline{V}\subseteq U_\alpha$ and choose an  equicontinuous pointwise bounded subset $A$ of $C(X)$ such that $A^\circ \subseteq V$. Clearly, the absolutely convex hull $\mathrm{acx}(A)$ of $A$ is also an  equicontinuous pointwise bounded subset of $C(X)$. So, by Fact \ref{f:Ascoli-Free-Ck}, the closure  $K:=\overline{\mathrm{acx}(A)}^{\,\tau_k}$ of $\mathrm{acx}(A)$ in the compact-open topology $\tau_k$ is a compact equicontinuous subset of $\CC(X)$. Since the bounded convex subsets of $C(X)$ in $\tau_k$ and $\sigma\big( C(X),M_c(X)\big)$ are the same, the Bipolar theorem implies that $K=K^{\circ\circ}$. As $$C_\alpha \subseteq \overline{V}^\circ \subseteq A^{\circ\circ} = K^{\circ\circ} =K$$  we obtain that $C_\alpha$ is compact.

Let $C$ be a compact subset of $\CC(X)$. Since $X$ is Ascoli, $C$ is equicontinuous and clearly pointwise bounded. Take $\alpha\in\NN^\NN$ such that $U_\alpha \subseteq C^\circ$. Then $$C\subseteq C^{\circ\circ} \subseteq U_\alpha^\circ =C_\alpha.$$ Thus the family $\mathcal{C}$ swallows the compact sets of $\CC(X)$.
\end{proof}

For a  space $X$ we denote by $\mu X$ the Dieudonn\'{e} completion of $X$. Note that any paracompact space is  Dieudonn\'{e} complete, see \cite[8.5.13(d)]{Eng}. Now Theorem \ref{t:Free-G-base-Main} immediately follows from the following  more general result.
\begin{theorem} \label{t:Free-G-base}
Let $X$ be a Tychonoff space such that $\mu X$ is an Ascoli space. Then  $L(X)$ has a $\GG$-base if and only if $\CC(\mu X)$ has a compact resolution swallowing  compact subsets. In this case the space $\CC(\mu X)$ is Lindel\"{o}f.
\end{theorem}
\begin{proof}
The space $L(X)$ has a $\GG$-base if and only if its (Raikov) completion $\overline{L(X)}$ has a $\GG$-base, see Proposition 2.7 of \cite{GKL}. It is known that $\overline{L(X)}$ is $(M_c(\mu X),\tau_e)$, see Theorem 5 of \cite{Usp2}. Now Proposition \ref{p:G-L-Mc} applies.
To prove the last assertion we note that the space $\CC(\mu X)$ is $K$-analytic by  Theorem 9.9 of \cite{kak}. So $\CC(\mu X)$ is  Lindel\"{o}f by Proposition 3.13 of \cite{kak}.
\end{proof}
We do not know whether the condition on $X$ to be an Ascoli space is essential in Theorem \ref{t:Free-G-base-Main}.

Question 4.18 in \cite{GKL}  asks whether for a $k$-space the existence of a $\GG$-base on $L(X)$ implies that also $\CC(\CC(X))$ has a $\GG$-base. By Ferrando--K{\c{a}}kol theorem \cite{fe-ka} (see also \cite[Theorem 4.9]{GKL}), the space $\CC(X)$ has a compact resolution swallowing compact subsets if and only if $\CC(\CC(X))$ has a $\GG$-base. Combining this result with Theorem \ref{t:Free-G-base-Main} we obtain a partial answer to \cite[Question 4.18]{GKL}.
\begin{corollary}
Let $X$ be a Dieudonn\'{e} complete Ascoli space. Then $L(X)$ has a $\GG$-base if and only if the space $\CC(\CC(X))$ has a $\GG$-base.
\end{corollary}

Corollary \ref{c:Free_G-base-countable} follows from the next result.
\begin{corollary} \label{c:Free-base-countable}
If $X$ is a countable Ascoli space, then the following assertions are equivalent:
\begin{enumerate}
\item[(i)] $L(X)$ has a $\GG$-base;
\item[(ii)] $\CC(X)$ has a compact resolution swallowing the compact sets of $\CC(X)$;
\item[(iii)] $X$ has a $\GG$-base.
\end{enumerate}
\end{corollary}
\begin{proof}
(i)$\Leftrightarrow$(ii) follows from Theorem \ref{t:Free-G-base-Main} (recall that any countable space being Lindel\"{o}f is Dieudonn\'{e} complete). (i)$\Rightarrow$(iii) follows from the fact that $X$ is a subspace of $L(X)$. (iii)$\Rightarrow$(ii) follows from Propositions \ref{p:Free-LCS-G-base} and \ref{p:Free-G-base-uniform}.
\end{proof}
Note that Ferrando in \cite{Ferrando-note}  gives a direct proof of the implication (iii)$\Rightarrow$(ii) in Corollary \ref{c:Free-base-countable}.


We provide another necessary  condition for a  space $X$ to have the space $L(X)$ with a $\GG$-base.
\begin{proposition}\label{p:L(X)-G-base-necessary}
If $L(X)$ has a $\GG$-base,  then every precompact set in $L(X)$ (hence also in $X$) is metrizable.
\end{proposition}
\begin{proof}
Note that in every locally convex space $E$ with a $\GG$-base every precompact set is metrizable, see \cite[Theorem 11.1]{kak}.  We conclude the proof by noticing that $X$ embeds into $L(X)$.
\end{proof}


Recall that a topological space $X$ is a {\em $k_{\omega}$-space} (an {\em $\mathcal{MK}_\omega$-space}) if $X$ is the inductive limit of a countable family of compact (compact and metrizable) subsets. We proved in  \cite{GK-GMS2} that $L(X)$ has a $\GG$-base for every $\mathcal{MK}_\omega$-space $X$. Combining this result with Proposition \ref{p:L(X)-G-base-necessary} we obtain
\begin{corollary}\label{co}
Let $X$ be a $k_{\omega}$-space. Then $L(X)$  has $\GG$-base if and only if $X$ is an $\mathcal{MK}_\omega$-space.
\end{corollary}

\begin{remark} {\em
In  \cite[Question 4.19]{GKL} we ask whether the existence of a $\GG$-base in the free abelian group $A(X)$ over a space $X$ implies that  $L(X)$ has also a $\GG$-base. Let $X$ be a discrete space. Then it is clear that $A(X)$ being discrete has a $\GG$-base. In \cite{LPT} it is shown that if $X$ is of cardinality $\geq\mathfrak{c}$, then  $L(X)$ does not have a $\GG$-base. This answers Question 4.19 of \cite{GKL} in the negative.  Our Corollary \ref{c:Free_G-base-metric} implies a stronger result: for every uncountable discrete space $X$, the space $L(X)$ does not have a $\GG$-base.  }
\end{remark}

\bibliographystyle{amsplain}

\end{document}